\documentclass[reqno]{amsart}

\usepackage[latin1]{inputenc}
\usepackage{amssymb}
\usepackage{graphicx}
\usepackage{amscd}
\usepackage[hidelinks]{hyperref}
\usepackage{color}
\usepackage{float}
\usepackage{graphics,amsmath,amssymb}
\usepackage{amsthm}
\usepackage{amsfonts}
\usepackage{latexsym}
\usepackage{epsf}
\usepackage{enumerate}
\usepackage{xifthen}
\usepackage{mathrsfs}
\usepackage{dsfont}
\usepackage{makecell}
\usepackage[FIGTOPCAP]{subfigure}
\usepackage{amsmath}
\allowdisplaybreaks[4]
\usepackage{listings}
\usepackage{etoolbox}
\usepackage{fancyhdr}
\usepackage{pdflscape}
\usepackage[title,toc,titletoc]{appendix}
\usepackage{enumitem}

 \newtheoremstyle{mytheorem}% name % cf. thmtest.tex of AMSLaTeX
 {3pt}%      Space above
 {3pt}%      Space below
 {\slshape}% Body font
 {}%         Indent amount (empty = no indent,
 {\bfseries}% Thm head font
 {.}%        Punctuation after thm head
 { }%        Space after thm head: " " = normal interword space;
 {}%         Thm head spec (can be left empty, meaning `normal')

\numberwithin{equation}{section}

\theoremstyle{theorem}
\newtheorem{theorem}{Theorem}[section]

\newtheorem{lemma}[theorem]{Lemma}

\theoremstyle{definition}

\newcommand{\Keywords}[1]{\ifthenelse{\isempty{#1}}{}{\smallskip \smallskip \noindent \textbf{Keywords}. #1}}
\newcommand{\MSC}[2][2010]{\ifthenelse{\isempty{#2}}{}{\smallskip \smallskip \noindent \textbf{#1MSC}. #2}}
\newcommand{\abstractnote}[1]{\ifthenelse{\isempty{#1}}{}{\smallskip \smallskip \noindent \textsuperscript{\dag}#1}}

\makeatletter
\def\specialsection{\@startsection{section}{1}%
  \z@{\linespacing\@plus\linespacing}{.5\linespacing}%
  {\normalfont}}% NEW
\def\section{\@startsection{section}{1}%
  \z@{.7\linespacing\@plus\linespacing}{.5\linespacing}%
  {\normalfont\scshape}}% NEW
\patchcmd{\@settitle}{\uppercasenonmath\@title}{\Large\boldmath}{}{}
\patchcmd{\@settitle}{\begin{center}}{\begin{flushleft}}{}{}
\patchcmd{\@settitle}{\end{center}}{\end{flushleft}}{}{}
\patchcmd{\@setauthors}{\MakeUppercase}{\normalsize}{}{}
\patchcmd{\@setauthors}{\centering}{\raggedright}{}{}
\patchcmd{\section}{\scshape}{\large\bfseries\boldmath}{}{}
\patchcmd{\subsection}{\bfseries}{\bfseries\boldmath}{}{}
\renewcommand{\@secnumfont}{\bfseries}
\patchcmd{\@startsection}{\@afterindenttrue}{\@afterindentfalse}{}{}
\patchcmd{\abstract}{\leftmargin3pc}{\leftmargin1pc}{}{}

\def\maketitle{\par
  \@topnum\z@ % this prevents figures from falling at the top of page 1
  \@setcopyright
  \thispagestyle{empty}% this sets first page specifications
  \ifx\@empty\shortauthors \let\shortauthors\shorttitle
  \else \andify\shortauthors
  \fi
  \@maketitle@hook
  \begingroup
  \@maketitle
  \toks@\@xp{\shortauthors}\@temptokena\@xp{\shorttitle}%
  \toks4{\def\\{ \ignorespaces}}% defend against questionable usage
  \edef\@tempa{%
    \@nx\markboth{\the\toks4
      \@nx\MakeUppercase{\the\toks@}}{\the\@temptokena}}%
  \@tempa
  \endgroup
  \c@footnote\z@
  \@cleartopmattertags
}
\makeatother

\newcommand{\pa}{\mathit{pa}}
\newcommand{\pac}{\mathcal{PA}}

\title[Unlimited parity alternating partitions]{Unlimited parity alternating partitions}

\author[S. Chern]{Shane Chern}
\address{Department of Mathematics, The Pennsylvania State University, University Park, PA 16802, USA}
\email{shanechern@psu.edu}

\date{}

\begin{document}

%{\footnotesize\noindent \textit{Preprint}, \arxiv{1802.04344}.}
%
%\bigskip \bigskip

\maketitle

\begin{abstract}
We introduce a new type of partitions that consists of partitions whose different parts alternate in parity (e.g., $3+2+2+1+1$). Various properties of this partition function are studied. In particular, we obtain its asymptotic behavior by employing Ingham's Tauberian theorem.


\MSC{05A17, 11P82.}
\end{abstract}

\section{Introduction}

A \textit{partition} of a positive integer $n$ is a weakly decreasing sequence of positive integers whose sum is $n$. Let $p(n)$ count the number of partitions of $n$. It is well known that
$$1+\sum_{n\ge 1}p(n)q^n=\frac{1}{(q;q)_\infty}.$$
Here and in the sequel, we adopt the standard $q$-series notations:
\begin{align*}
(a;q)_n&:=\prod_{k=0}^{n-1}(1-aq^{k}),\\
(a;q)_{\infty}&:=\prod_{k= 0}^\infty (1-aq^{k}),\\
(a_1,a_2,\cdots,a_m;q)_n&:=(a_1;q)_n(a_2;q)_n\cdots(a_m;q)_n,\\
(a_1,a_2,\cdots,a_m;q)_\infty&:=(a_1;q)_\infty(a_2;q)_\infty\cdots(a_m;q)_\infty.
\end{align*}

In recent years, many authors also studied partitions with further restrictions. One example is the \textit{parity alternating partitions} introduced by Andrews \cite{And1984}. Here a parity alternating partition is a partition in which the parts alternate in parity. If we further require the smallest part of parity alternating partitions to be odd, then Andrews showed that this type of partitions has generating function
$$\sum_{n\ge 0}\frac{q^{n(n+1)/2}}{(q^2;q^2)_n}.$$
Jang \cite{Jan2017} later studied the asymptotic behavior of the number of parity alternating partitions of $n$ with the smallest part odd (in her paper, this type of partitions is called \textit{odd-even partitions}), and proved that the number is asymptotic to
$$\frac{1}{2\sqrt{5}n^{\frac{3}{4}}}e^{\pi \sqrt{\frac{n}{5}}}$$
for sufficiently large $n$.

It is easy to observe that parity alternating partitions are distinct partitions (i.e.~all parts are distinct). Naturally, we may study an unlimited version of parity alternating partitions. Here an \textit{unlimited parity alternating partition} is a partition whose different parts alternate in parity. In other words, we allow partitions like $3+2+2+1+1$. For example, $6$ has the following eight unlimited parity alternating partitions:
\begin{gather*}
6,\ 4+1+1,\ 3+3,\ 3+2+1,\ 2+2+2,\\
2+2+1+1,\ 2+1+1+1+1,\ 1+1+1+1+1+1.
\end{gather*}

Let $\pac$ be the set of unlimited parity alternating partitions. For a given positive integer $n$, we use $\pac(n)$ to denote the set of partitions of $n$ in $\pac$. We further write $\pa(n)=|\pac(n)|$, the number of unlimited parity alternating partitions of $n$. We remark that $\pa(n)$ is sequence A242110 in OEIS \cite{OEIS}.

The goal of this paper is to study various properties of $\pa(n)$. In particular, our core result is the following asymptotic formula:

\begin{theorem}\label{th:main}
We have, as $n\to\infty$,
\begin{equation}
\pa(n)\sim \frac{\sqrt{A}}{2\pi n}e^{2\sqrt{An}},
\end{equation}
where
$$A=\frac{\pi^2}{12}+2\left(\log \frac{1+\sqrt{5}}{2}\right)^2.$$
\end{theorem}

\section{Monotonicity}

We start by listing the first 15 values of $\pa(n)$, which are shown in Table \ref{ta:pa}.

\begin{table}[ht]\caption{The first 15 values of $\pa(n)$}\label{ta:pa}
\centering
\begin{tabular}{|c|c||c|c||c|c|}
$n$ & $\pa(n)$ & $n$ & $\pa(n)$ & $n$ & $\pa(n)$\\
$1$ & $1$ & $6$ & $8$ & $11$ & $33$\\
$2$ & $2$ & $7$ & $11$ & $12$ & $39$\\
$3$ & $3$ & $8$ & $13$ & $13$ & $54$\\
$4$ & $4$ & $9$ & $21$ & $14$ & $63$\\
$5$ & $6$ & $10$ & $23$ & $15$ & $88$\\
\end{tabular}
\end{table}

From these values, one may expect that $\{\pa(n)\}_{n\ge 1}$ is a strictly increasing sequence. In fact, this observation is correct.

\begin{theorem}\label{th:inc}
$\{\pa(n)\}_{n\ge 1}$ is a strictly increasing sequence.
\end{theorem}

\begin{proof}
We may assume that $n>12$. For the remaining cases, one may check directly through Table \ref{ta:pa}.

The main idea of our proof is to find a injective map $\phi_{n}:\pac(n)\to\pac(n+1)$ for each $n$ such that $\phi_n(\pac(n))$ is a proper subset of $\pac(n+1)$. Let $\lambda\in\pac(n)$. We define the map $\phi_n$ as follows:
\begin{itemize}[leftmargin=\parindent]
\item If $\lambda=(\lambda_1,\ldots,\lambda_\ell)$ with $\lambda_\ell$ even, then
$$\phi_n(\lambda)=(\lambda_1,\ldots,\lambda_\ell,1).$$
Notice that $\phi_n(\lambda)$ ends with one $1$.
\item If $\lambda=(\lambda_1,\ldots,\lambda_\ell)$ with $\lambda_\ell=1$, then
$$\phi_n(\lambda)=(\lambda_1,\ldots,\lambda_\ell,1).$$
Notice that $\phi_n(\lambda)$ ends with at least two $1$'s.
\item If $\lambda=(\lambda_1,\ldots,\lambda_\ell)$ with $\lambda_{\ell-1}=\lambda_\ell\ge 5$ odd, then
$$\phi_n(\lambda)=(\lambda_1,\ldots,\lambda_{\ell-1},\underbrace{2,\ldots,2}_{(\lambda_\ell+1)/2}).$$
Notice that $\phi_n(\lambda)$ ends with an odd integer $\lambda_{\ell-1}\ge 5$ and at least three $2$'s.
\item If $\lambda=(\lambda_1,\ldots,\lambda_\ell)$ with $\lambda_{\ell-1}\ne\lambda_\ell$ and $\lambda_\ell\ge 5$ odd, then we study it into three cases:
\begin{enumerate}
\item If $\lambda_\ell\equiv -1 \pmod{3}$, then
$$\phi_n(\lambda)=(\lambda_1,\ldots,\lambda_{\ell-1},\underbrace{3,\ldots,3}_{(\lambda_\ell+1)/3}).$$
Notice that $\phi_n(\lambda)$ ends with an even integer $\lambda_{\ell-1}\ge 6$ and at least two $3$'s. (In this case, if $\lambda=(\lambda_1)$ has only one part, then $\phi_n(\lambda)$ has merely $3$ as its parts and there are at least two $3$'s.)
\item If $\lambda_\ell\equiv 0 \pmod{3}$, then
$$\phi_n(\lambda)=(\lambda_1,\ldots,\lambda_{\ell-1},\underbrace{3,\ldots,3}_{(\lambda_\ell-3)/3},2,2).$$
Notice that $\phi_n(\lambda)$ ends with two $2$'s and at least two $3$'s.
\item If $\lambda_\ell\equiv 1 \pmod{3}$, then
$$\phi_n(\lambda)=(\lambda_1,\ldots,\lambda_{\ell-1},\underbrace{3,\ldots,3}_{(\lambda_\ell-1)/3},2).$$
Notice that $\phi_n(\lambda)$ ends with one $2$ and at least two $3$'s.
\end{enumerate}
\item If $\lambda=(\lambda_1,\ldots,\lambda_\ell)$ with $\lambda_\ell=3$, then we study it into four cases:
\begin{enumerate}
\item If $\lambda=(\lambda_1,\ldots,\lambda_{\ell-4},3,3,3,3)$, then
$$\phi_n(\lambda)=(\lambda_1,\ldots,\lambda_{\ell-4},3,2,2,2,2,2).$$
Notice that $\phi_n(\lambda)$ ends with five $2$'s and at least one $3$.
\item If $\lambda=(\lambda_1,\ldots,\lambda_{\ell-3},3,3,3)$ with $\lambda_{\ell-3}\ne 3$, then
$$\phi_n(\lambda)=\begin{cases}
(\lambda_1,\ldots,\lambda_{\ell-3},5,5) & \text{if $4$ is not a part of $\lambda$},\\
(\lambda_1,\ldots,\lambda_{\ell-3},4,3,3) & \text{if $4$ is a part of $\lambda$}.
\end{cases}$$
Notice that $\phi_n(\lambda)$ ends with two $5$'s in the first case, and two $3$'s and at least two $4$'s in the second case.
\item If $\lambda=(\lambda_1,\ldots,\lambda_{\ell-2},3,3)$ with $\lambda_{\ell-2}\ne 3$, then
$$\phi_n(\lambda)=(\lambda_1,\ldots,\lambda_{\ell-2},3,2,2).$$
Notice that $\phi_n(\lambda)$ ends with two $2$'s and one $3$.
\item If $\lambda=(\lambda_1,\ldots,\lambda_{\ell-1},3)$ with $\lambda_{\ell-1}\ne 3$, then
$$\phi_n(\lambda)=\begin{cases}
(\lambda_1+1,\lambda_2,\ldots,\lambda_{\ell-1},3) & \text{if $\lambda_1=\lambda_2$},\\
(\lambda_1+4,\lambda_2,\ldots,\lambda_{\ell-1}) & \text{if $\lambda_1\ne\lambda_2$}.
\end{cases}$$
Notice that $\phi_n(\lambda)$ ends with one $3$ in the first case, and an even integer $\ge 4$ in the second case.
\end{enumerate}
\end{itemize}

One may check that for all $n> 12$, the map $\phi_n$ is one-to-one. To see $\pa(n)<\pa(n+1)$, we only need to find a partition in $\pac(n+1)$ with no pre-image under $\phi_n$. This is trivial as the partition $(2,2,\ldots,2)$ for even $n\ge 14$ and the partition $(3,2,2,\ldots,2)$ for odd $n\ge 15$ have no pre-image under $\phi_{n-1}$.

Hence $\{\pa(n)\}_{n\ge 1}$ is a strictly increasing sequence.
\end{proof}

\section{Generating function}

To study the generating function of $\pa(n)$, we first turn to partitions in which all different parts except for the largest one appear an odd number (or zero) times. Note that the conjugate of any partition in this partition set is in $\pac$ and vice versa. For example (here $\bar{\lambda}$ denotes the conjugate of $\lambda$),
\begin{align*}
&\lambda=(4,3,3,3,1),& \bar{\lambda}=(5,4,4,1)\in\pac;\\
&\lambda=(4,4,3,1,1,1),& \bar{\lambda}=(6,3,3,2)\in\pac.
\end{align*}

As a consequence, suppose that $p_o^*(n)$ counts the number of the aforementioned partitions of $n$, then
\begin{theorem}
We have
\begin{equation}
\pa(n)=p_o^*(n).
\end{equation}
\end{theorem}

At last, we notice that the generating function of $p_o^*(n)$ is easy to write. This leads to
\begin{theorem}
We have
\begin{equation}
\sum_{n\ge 1}\pa(n)q^n=\sum_{n\ge 1}\frac{q^n}{1-q^{n}}\prod_{k=1}^{n-1}\left(1+\frac{q^k}{1-q^{2k}}\right).
\end{equation}
\end{theorem}

\begin{proof}
This directly comes from
\begin{align*}
\sum_{n\ge 1}\pa(n)q^n=\sum_{n\ge 1}p_o^*(n)q^n&=\sum_{n\ge 1}(q^n+q^{2n}+\cdots)\prod_{k=1}^{n-1}(1+q^k+q^{3k}+\cdots)\\
&=\sum_{n\ge 1}\frac{q^n}{1-q^{n}}\prod_{k=1}^{n-1}\left(1+\frac{q^k}{1-q^{2k}}\right).
\end{align*}
\end{proof}

\section{Asymptotic behavior}

\subsection{More about the generating function}

In order to study the asymptotic behavior of $\pa(n)$, we need to rewrite the generating function of $\pa(n)$ to make it easier to use Ingham's Tauberian theorem.

\begin{theorem}\label{th:gf1}
We have
\begin{equation}
\sum_{n\ge 1}\pa(n)q^n=\frac{3-\sqrt{5}}{2}\left(\prod_{k\ge 1}\frac{1+q^k-q^{2k}}{1-q^{2k}}\right)\left(1+\frac{3+\sqrt{5}}{2}\sum_{n\ge 1}\frac{\left(\frac{\sqrt{5}-1}{2}\right)^n}{1+\frac{\sqrt{5}+1}{2}q^n}\right)-2.
\end{equation}
\end{theorem}

Our proof relies on Heine's second transformation.

\begin{lemma}[{\cite[Eq.~(1.4.5)]{GR2004}}]
We have, for $|z|<1$ and $|c|<|b|$,
\begin{equation}
\sum_{n\ge 0}\frac{(a,b;q)_n}{(q,c;q)_n}z^n=\frac{(c/b,bz;q)_\infty}{(c,z;q)_\infty}\sum_{n\ge 0}\frac{(abz/c,b;q)_n}{(q,bz;q)_n}\left(\frac{c}{b}\right)^n.
\end{equation}
\end{lemma}

\begin{proof}[Proof of Theorem \ref{th:gf1}]
Note that
$$1+\frac{q^k}{1-q^{2k}}=\frac{\left(1-\frac{\sqrt{5}-1}{2}q^k\right)\left(1+\frac{\sqrt{5}+1}{2}q^k\right)}{(1-q^k)(1+q^k)}$$
and
$$\left(1-\frac{\sqrt{5}-1}{2}\right)\left(1+\frac{\sqrt{5}+1}{2}\right)=1.$$

We have
\begin{align*}
\sum_{n\ge 1}\pa(n)q^n&=\sum_{n\ge 1}\frac{q^n}{1-q^{n}}\prod_{k=1}^{n-1}\left(1+\frac{q^k}{1-q^{2k}}\right)\\
&=\sum_{n\ge 1}\frac{q^n}{1-q^{n}}\prod_{k=1}^{n-1}\frac{\left(1-\frac{\sqrt{5}-1}{2}q^k\right)\left(1+\frac{\sqrt{5}+1}{2}q^k\right)}{(1-q^k)(1+q^k)}\\
&=2\sum_{n\ge 1}\frac{\left(\frac{\sqrt{5}-1}{2},-\frac{\sqrt{5}+1}{2};q\right)_n}{(q,-1;q)_n}q^n\\
&=2\sum_{n\ge 0}\frac{\left(\frac{\sqrt{5}-1}{2},-\frac{\sqrt{5}+1}{2};q\right)_n}{(q,-1;q)_n}q^n-2\\
&=2\ \frac{\left(\frac{\sqrt{5}-1}{2},-\frac{\sqrt{5}+1}{2}q;q\right)_\infty}{(-1,q;q)_\infty}\sum_{n\ge 0}\frac{\left(-\frac{\sqrt{5}+1}{2};q\right)_n}{\left(-\frac{\sqrt{5}+1}{2}q;q\right)_n}\left(\frac{\sqrt{5}-1}{2}\right)^n-2 \tag{by Heine's Second Tranformation}\\
&=\frac{3-\sqrt{5}}{2}\left(\prod_{k\ge 1}\frac{1+q^k-q^{2k}}{1-q^{2k}}\right)\left(1+\frac{3+\sqrt{5}}{2}\sum_{n\ge 1}\frac{\left(\frac{\sqrt{5}-1}{2}\right)^n}{1+\frac{\sqrt{5}+1}{2}q^n}\right)-2.
\end{align*}
\end{proof}

\subsection{Ingham's Tauberian theorem}

Ingham's Tauberian theorem is a powerful tool to determine asymptotic behaviors of certain weakly increasing nonnegative sequences. It states as follows

\begin{theorem}[Ingham \cite{Ing1941}]\label{th:ing}
Let $f(q)=\sum_{n\ge0}a(n) q^n$ be a power series with weakly increasing nonnegative coefficients and radius of convergence equal to $1$. If there are constants $A>0$ and $\lambda,\alpha\in\mathbb{R}$ such that
$$f\left(e^{-\epsilon}\right) \sim\lambda\epsilon^\alpha e^{\frac{A}{\epsilon}}$$
as $\epsilon\to 0^+$, then
$$a(n)\sim \frac{\lambda}{2\sqrt{\pi}}\frac{A^{\frac{\alpha}{2}+\frac14}}{n^{\frac{\alpha}{2}+\frac34}} e^{2\sqrt{An}}$$
as $n\to \infty$.
\end{theorem}

\subsection{Proof of Theorem \ref{th:main}}

We first notice from Theorem \ref{th:inc} that $\{\pa(n)\}_{n\ge 1}$ is a strictly increasing positive sequence.

Let $PA(q):=\sum_{n\ge 1}\pa(n)q^n$. It remains to estimate $PA(e^{-\epsilon})$ as $\epsilon\to 0^+$. We recall two known results.

The first result tells us the asymptotic behavior of $(e^{-2\epsilon};e^{-2\epsilon})_\infty$. The modular inversion formula for Dedekind's eta-function (p.~121, Proposition 14 of \cite{Kob1984}) implies that as $\epsilon\to 0^+$
$$(e^{-\epsilon};e^{-\epsilon})_\infty\sim \sqrt{\frac{2\pi}{\epsilon}}e^{-\frac{\pi^2}{6\epsilon}}.$$

On the other hand, Auluck et al.~\cite{ASA1950} showed that as $\epsilon\to 0^+$
$$\prod_{k\ge 1}\Big(1+e^{-k\epsilon}-e^{-2k\epsilon}\Big)\sim e^{\frac{2}{\epsilon}\left(\log \frac{1+\sqrt{5}}{2}\right)^2}.$$
It is worth pointing out that in the original paper of Auluck et al., they did not use the number $2\left(\log \frac{1+\sqrt{5}}{2}\right)^2$ in the exponent. Instead, they used its decimal value $0.46313\cdots$. However, it is clear from their paper that the value is
$$\int_0^1 \frac{\log (1+x-x^2)}{x}dx,$$
which is indeed $2\left(\log \frac{1+\sqrt{5}}{2}\right)^2$.

At last, we have (with $q=e^{-\epsilon}$) as $\epsilon\to 0^+$
$$1+\frac{3+\sqrt{5}}{2}\sum_{n\ge 1}\frac{\left(\frac{\sqrt{5}-1}{2}\right)^n}{1+\frac{\sqrt{5}+1}{2}q^n}=\frac{3+\sqrt{5}}{2}+O(\epsilon).$$
Here $(3+\sqrt{5})/2$ comes from taking $q=1$.

Combining all these ingredients together, we conclude that
$$PA(e^{-\epsilon})\sim \sqrt{\frac{\epsilon}{\pi}}e^{\frac{1}{\epsilon}\left(\frac{\pi^2}{12}+2\left(\log \frac{1+\sqrt{5}}{2}\right)^2\right)}.$$

In the setting of Ingham's Tauberian theorem, we have
$$\lambda=\sqrt{\frac{1}{\pi}},\qquad \alpha=\frac{1}{2},\qquad A=\frac{\pi^2}{12}+2\left(\log \frac{1+\sqrt{5}}{2}\right)^2.$$
It therefore follows from Ingham's Tauberian theorem that
$$pa(n)\sim \frac{\sqrt{A}}{2\pi n}e^{2\sqrt{An}}$$
as $n\to\infty$.

\subsection{Further remarks}

Let $\pa_o(n)$ denote the number of unlimited parity alternating partitions of $n$ with the smallest part odd. We observe that the conjugate of such partitions are partitions in which all different parts appear an odd number (or zero) times and vice versa. Hence $\pa_o(n)$ has a more succient generating function.

\begin{theorem}
We have
\begin{equation}
1+\sum_{n\ge 1}\pa_o(n)q^n=\prod_{k\ge 1}\frac{1+q^k-q^{2k}}{1-q^{2k}}.
\end{equation}
\end{theorem}

We know from Auluck et al.~\cite{ASA1950} that $\pa_o(n)$ is also asymptotic to 
$$\frac{\sqrt{A}}{2\pi n}e^{2\sqrt{An}}\qquad(n\to\infty),$$
with $A$ defined in the previous section. This tells us that the partition set $\pac$ is dominated by partitions with the smallest part odd.

\subsection*{Acknowledgements}

I would like to thank George E. Andrews and Robert C. Vaughan for some helpful discussions.

\bibliographystyle{amsplain}

\end{document}